\documentclass{article}

\usepackage[a4paper, textwidth=14cm, textheight= 21cm]{geometry}
\usepackage{amsmath,amssymb,enumerate,amsthm}

\pdfoutput=1

\usepackage[colorlinks=true]{hyperref}
\usepackage{csquotes}
\usepackage{enumitem}

\usepackage{url}
 
\newtheorem{theorem}{Theorem}[section]
\newtheorem{corollary}[theorem]{Corollary}
\newtheorem{proposition}[theorem]{Proposition}
\theoremstyle{definition}

\mathchardef\ordinarycolon\mathcode`\:
\mathcode`\:=\string"8000\def\R{\mathbb{R}}
\begingroup \catcode`\:=\active
  \gdef:{\mathrel{\mathop\ordinarycolon}}
\endgroup

\renewcommand{\R}{\mathbb{R}}

\newcommand{\rr}{\mathbb{R}}

\usepackage{mathrsfs}

\newcommand{\Ac}{\mathcal{A}}

\newcommand{\Dc}{\mathcal{D}}
\newcommand{\Ec}{\mathcal{E}}
\newcommand{\Fc}{\mathcal{F}}

\newcommand{\Mc}{\mathcal{M}}

\newcommand{\Qc}{\mathcal{Q}}

\newcommand{\Sc}{\mathcal{S}}

\newcommand{\WF}{\mathrm{WF}}                         
\newcommand{\norm}[1]{\left\lVert#1\right\rVert}      
\newcommand{\abs}[1]{\left|#1\right|}                 
\newcommand{\paren}[1]{\left(#1\right)}               
\newcommand{\bparen}[1]{\left[#1\right]}               
\newcommand{\sparen}[1]{\left\{#1\right\}}		      

\renewcommand{\d}{\,\mathrm{d}}						  

\DeclareMathOperator*{\supp}{\mathrm{supp}}           
\renewcommand{\epsilon}{\varepsilon}
\renewcommand{\rho}{\varrho}
\newcommand{\vp}{\varphi}
\newcommand{\thperp}{{\theta^\bot}}

\renewcommand{\tilde}{\widetilde}

\newcommand{\rtwo}{{{\mathbb R}^2}}

\newcommand{\rn}{{{\mathbb R}^n}}

\newcommand{\st}{\hskip 0.3mm : \hskip 0.3mm}
\renewcommand{\th}{\theta}

\newcommand{\be}{\begin{equation}}
\newcommand{\ee}{\end{equation}}
\newcommand{\bea}{\begin{eqnarray}}
\newcommand{\eea}{\end{eqnarray}}
\newcommand{\bean}{\begin{eqnarray*}}
\newcommand{\eean}{\end{eqnarray*}}

\newcommand{\bel}[1]{\begin{equation}\label{#1}}

\newcommand{\eel}[1]{{\label{#1}\end{equation}}}

\newcommand{\inv}{^{-1}}

\newcommand{\sxr}{{S^1\times \mathbb{R}}}

\newcommand{\smo}{\setminus\boldsymbol{0}}

\newcommand{\xio}{{\xi_0}}

\newcommand{\chiab}{\chi_{{(a,b)}}}
\newcommand{\chia}{\chi_{A}}
\newcommand{\dx}{\mathbf{dx}}

\newcommand{\dr}{\mathbf{dr}}
\newcommand{\ds}{\mathbf{ds}}

\newcommand{\dphi}{\mathbf{d\phi}}
\newcommand{\Om}{\Omega}

\newcommand{\Vc}{\mathcal{V}}
\newcommand{\Vab}{\mathcal{V}_{(a,b)}}

\newcommand{\phio}{\phi_0}
\newcommand{\phion}{\phi_1}

\newcommand{\lamo}{\lambda_0}
\newcommand{\lamon}{\lambda_1}

\newcommand{\Wab}{W_{\{a,b\}}}

\newcommand{\Rmu}{R_\mu}
\newcommand{\Rmuab}{R_{\mu,(a,b)}}
\newcommand{\Rstm}{R^*_\mu}
\newcommand{\Kvp}{\mathcal{K}_\varphi}

\newcommand{\Lvp}{\mathcal{L}_{\varphi}}

\newcommand{\Ma}{\mathcal{M}_\mathrm{A}}
\newcommand{\Mst}{\mathcal{M}^\ast}

\newcommand{\tx}{\tilde{x}}

\title{A paradigm for the characterization of artifacts in tomography} 

\author{J\"urgen Frikel\footnotemark[1] \and  Eric Todd  Quinto\footnotemark[2]}

\date{}

\begin{document}
\maketitle

\renewcommand{\thefootnote}{\fnsymbol{footnote}}
\footnotetext[1]{Zentrum Mathematik - M6, Technische Universit\"at M\"unchen, Germany and Institute of Computational Biology, Helmholtz Zentrum M\"unchen, Germany, \textbf{Email:}~\texttt{juergen.frikel@helmholtz-muenchen.de}}
\footnotetext[2]{Department of Mathematics, Tufts University, Medford, MA 02155, USA; \textbf{Email:}~\texttt{todd.quinto@tufts.edu}}

\begin{abstract}
We present a paradigm for characterization of artifacts in limited
data tomography problems. In particular, we use this paradigm to
characte\-rize artifacts that are generated in reconstructions from
limited angle data with generalized Radon transforms and general
filtered backprojection type operators. In order to find when visible
singularities are imaged, we calculate the symbol of our
reconstruction operator as a pseudodifferential operator.


\end{abstract}

\section{Introduction}

In this article, we consider the generalized Radon transform
integrating over lines in the plane.  Let $s\in \rr$,
$\phi\in[0,2\pi]$ and $\th(\phi) = (\cos(\phi),\sin(\phi))$ be the
unit vector in $S^1$ in direction $\phi$ and $\thperp(\phi)=(-\sin
(\phi),\cos(\phi))$, then $\thperp(\phi)$ is perpendicular to
$\th(\phi)$.  Let $\Xi = [0,2\pi]\times \rr$, then for each
$(\phi,s)\in \Xi$,  $L(\phi,s) = \sparen{x\in \rtwo\st
x\cdot\th(\phi) = s}$ is the line containing $s\th(\phi)$ and normal
to $\th(\phi)$.  We let $\mu(\phi,x)$ be a smooth function on
$\rr\times \rtwo$ that is $2\pi-$periodic in $\phi$.  Then, we define
the generalized Radon transform \bel{def:Rmu} \Rmu f(\phi,s)
=\int_{x\in L(\phi,s)}f(x)\mu(\phi,x)\,\d x \ee where $\d x$ denotes
the arc length measure on the line.  This transform integrates
functions along lines.

We define the corresponding dual transform (or the backprojection
operator) for $g\in \Sc(\sxr)$ as \bel{def:Rstar} \Rstm g(x)
=\int_0^{2\pi} g(\phi,x\cdot\th(\phi))\mu(\phi,x)\d\phi, \ee which is
the integral of $g$ over all lines through $x$ (since for each
$\th(\phi)$, $x\in L(\phi,x\cdot \th(\phi))$).  Note that authors,
including Beylkin and others, use the weight $1/\mu$ for a different
weighted dual operator. These transforms are both defined and weakly
continuous for classes of distributions \cite{He:RT2011}.  Many
inversion formulas have been proven for the classical Radon transform
($\mu\equiv 1$) \cite{Natterer86}, and invertibility of $\Rmu$ has
been well studied (e.g., \cite{Beylkin, 
Novikov:att-inversion,Quinto1980}). Among the most prominent
reconstruction formulas are those of filtered backprojection type
\cite{Beylkin,Natterer86,Kun:sphere} which have the following
 form 
\begin{equation}
\label{eq:reconstruction operators}
	B g=\Rstm P g, \ \text{ where }\  g=\Rmu f, 
\end{equation}
and $P$ is a pseudodifferential operator that \enquote{filters} the
data $g=\Rmu f$. For example, in case of the classical Radon transform
and $P=(1/4\pi)\sqrt{-\partial^2/\partial s^2}$,
\eqref{eq:reconstruction operators} is an exact reconstruction formula
and the basis for the filtered backprojection (FBP) algorithm
\cite{Natterer86}. Another prominent example is the so-called Lambda
reconstruction formula which uses the filter
$P=(1/4\pi)(-\partial^2/\partial s^2)$ in \eqref{eq:reconstruction
operators}.  

In this paper, we consider the problem of reconstructing $f$ from
incomplete data. More precisely, we assume that $\Rmu f(\phi,s)$ is
known only for a limited angular range $\phi\in(a,b)$ (note that for
$b-a\geq \pi$, every line can be parameterized by $\phi\in (a,b)$ although
for general $\mu$, the measure might be different on the line
$L(\phi,s)$ and $L(\phi+\pi,-s)$--$\mu(\phi,x)$ might not equal
$\mu(\phi+\pi,x)$ for all $(\phi,x)$). Thus, we deal with the restricted
(or limited angle) generalized Radon transform which we define as 
\begin{equation}
\label{eq:restricted Radon transform}
	\Rmuab f(\phi,s) = \chi_{(a,b)\times\R}(\phi,s)\cdot\Rmu f(\phi,s),
\end{equation}
where $\chi_{(a,b)\times\R}$ denotes the characteristic function of the
data space $(a,b)\times\R$ with the limited angular range $(a,b)$ with
$b-a<\pi$ (or $b-a<2\pi$ if $\mu$ is not symmetric.  Such limited angle
problems arise in many practical situations and the filtered backprojection
type reconstruction of the form \eqref{eq:reconstruction operators} is
still one of the preferred reconstruction methods
\cite{SidkyPanVannier2009}. It is well known that in this situation only
visible singularities can be reconstructed reliably \cite{Quinto93} and
that the reconstruction problem is severely ill-posed
\cite{Lo1986,Natterer86}. Moreover, it has been shown in
\cite{FrikelQuinto2013,Katsevich:1997} that additional artifacts can be
generated. In \cite{FrikelQuinto2013,Katsevich:1997}, the authors consider
the limited angle FBP and Lambda reconstructions for the classical limited
angle tomography data $g_{(a,b)}= R_{(a,b)} f$ (i.e.  $\mu\equiv 1$ and
$P=\sqrt{-d^2/ds^2}$ for FBP and $P=-d^2/ds^2$ for Lambda) and derive a
precise geometric characterization of artifacts. In particular, they show
artifacts are generated along straight lines that are tangent to
singularities of $f$ whose directions correspond to the ends of the angular
range.  In \cite{LNguyen1}, L. Nguyen characterized the strength of these
artifacts.  In \cite{FrikelQuinto2013, Katsevich:1997} the authors prove
that a simple artifact reduction strategy smooths the artifacts.  The same
reduction strategy is proposed in \cite{KLM} for $\Rmu$ and the Lambda and
FBP filter, and the symbols are calculated for those specific operators for
limited angle and ROI data.

The methods of \cite{FrikelQuinto2013,Katsevich:1997, LNguyen1} do not
directly apply to the limited angle problem for the generalized Radon
transform with reconstruction operators \eqref{eq:reconstruction
operators} (with $P$ being an arbitrary pseudodifferential operator).
This is mainly due to the fact that their proofs rely on explicit
expressions of the operators as singular pseudodifferential operators.

In this paper, we study the application of the reconstruction
operators \eqref{eq:reconstruction operators} to the limited angle
data for an arbitrary $\mu$ which is smooth and nowhere zero. Using
the framework of microlocal analysis and the calculus of Fourier
integral operators, we prove a qualitative characterization of
artifacts and provide an artifact reduction strategy. Our proofs use
the technique that was originally developed in \cite{FrikelQuinto2014}
to characterize artifacts in photoacoustic tomography and sonar. In
particular, we show that the visible and added singularities are
contained in the same set as was obtained for specific cases in
\cite{FrikelQuinto2013, Katsevich:1997}.  We show that the artifact
reduction strategy in \cite{FrikelQuinto2013, Katsevich:1997,KLM}
applies for general filters $P$ (Theorem \ref{thm:reduction1}) and we
show for some choices of $P$ that most of the visible singularities are
recovered by the artifact reduced reconstruction operator (Corollary
\ref{cor:reduction2}).  

The rest of the article is organized as follows. Basic definitions and notations
are given in Section \ref{sec:notation}. In Section \ref{sec:paradigm} we present a general paradigm to
characterize added singularities in limited view tomography. The
characterization of limited angle artifacts for the generalized Radon
transform is proven in Section \ref{sec:characterization}, and the
artifact reduction strategy and symbol calculations are presented in
Section \ref{sec:reduction}.

\section{Notation}
\label{sec:notation}

Let $\Om$ be an open set.  We denote the set of $C^\infty$ functions
with domain $\Om$, by $\Ec(\Om)$ and the set of $C^\infty$ functions
of compact support in $\Om$ by $\Dc(\Om)$.  Distributions are
continuous linear functionals on these function spaces.  The dual
space to $\Dc(\Om)$ is denoted $\Dc'(\Om)$ and the dual space to
$\Ec(\Om)$ is denoted $\Ec'(\Om)$. In fact, $\Ec'(\Om)$ is the set of
distributions of compact support in $\Om$.  For more information
about these spaces we refer to \cite{Rudin:FA}.

We will use the framework of microlocal analysis for our
characterizations. Here, the notion of a wavefront set of a
distribution $f\in\Dc'(\Om)$ is central. It simulta\-neously describes
the locations and directions of singularities of $f$. That is, $f$ has
a singularity at $x_0\in\Om$ in direction $\xio\in \rn\smo$ if for any
cutoff function $\vp$ at $x_0$, the Fourier transform $\Fc(\vp f)$
does not decay rapidly in any open conic neighborhood of the ray
$\{t\xio\st t>0\}$.  Then, the \emph{wavefront set} of
$f\in\Dc'(\Om)$, $\WF(f)$, is defined as the set of all tuples
$(x_0,\xio)$ such that $f$ is singular at $x_0$ in direction $\xio$.
As defined, $\WF(f)$, is a closed subset of $\rn\times(\rn\smo)$ that
is conic in the second variable. However, in what follows, we will
view the wavefront set as a subset of a cotangent bundle so it will be
invariantly defined on manifolds \cite{Treves:1980vf}.

We recall that, for a manifold $\Xi$ and $y\in \Xi$, the cotangent
space of $\Xi$ at $y$, $T^*_y(\Xi)$ is the vector space of all first
order differentials (the dual space to the tangent space $T_y(\Xi)$),
and the cotangent bundle $T^*(\Xi)$ is the vector bundle with fiber
$T^*_y(\Xi)$ above $y\in \Xi$.  That is $T^*(\Xi)=\sparen{(y,\eta)\st
y\in \Xi, \eta\in T^*_y(\Xi)}$. The differentials $\dx_1$,
$\dx_2,\dots$, and $\dx_n$ are a basis of $T^*_x(\rn)$ for any $x\in
\rn$.  For $\xi\in \rn$, we will use the notation \[\xi\dx =
\xi_1\dx_1+\xi_2\dx_2+\cdots+\xi_n\dx_n\in T^*_x(\rn).\] If $\phi\in
\rr$ then $\dphi$ will be the differential with respect to $\phi$, and differentials
$\dr$ and $\ds$ are defined analogously.  \

Let $X$ and $Y$ be manifolds, and $C\subset T^\ast (Y)\times T^\ast (X)$,
then
\begin{equation}
	C^t = \sparen{(x,\xi;y,\eta)\st  (y,\eta;x,\xi)\in
	C}.\label{def:At}
\end{equation}
If $D\subset T^*(X) $, we define
\bel{def:circ}
	C\circ D = \sparen{(y,\eta)\in T^\ast(Y)\st  \exists (x,\xi)\in
	D\st (y,\eta;x,\xi)\in C}.
\ee

Fourier integral operators (FIO) are linear operators on distribution
spaces that precisely transform wavefront sets.  They are defined in
\cite{Ho1971, Treves:1980vf} in terms of amplitudes and phase
functions.  If $X$ and $\Xi$ are manifolds and  $\Fc:\Dc'(X)\to
\Dc'(\Xi)$ is a FIO, then associated to $\Fc$ is the \emph{canonical
relation}  $C\subset T^*(\Xi)\times T^*(X)$.  Then the
H\"ormander-Sato Lemma (e.g., 
\cite[Th.\ 5.4, p.\
461]{Treves:1980vf}) asserts for $f\in \Ec'(X)$ that 
\begin{equation}
\label{thm:HS}
	\WF(\Fc f)\subset C\circ\WF(f).
\end{equation}

\section{The paradigm} \label{sec:paradigm} 

In this section, we will present a methodology that can be used to
prove characterizations of limited view artifacts for a number of
tomography problems.  In the next section, we will apply them to
$\Rmu$.

This methodology was originally developed in \cite{FrikelQuinto2014}
to understand visible and added singularities in limited data
photoacoustic tomography and sonar.  Denote the forward operator by
$\Mc:\Ec'(\Omega)\to\Ec'(\Xi)$ and assume $\Mc$ is a FIO.  The
\emph{object space} $\Omega$ is a region to be imaged and the
\emph{data space} $\Xi$ is a space that parameterizes the data.  A
\emph{limited data problem} for $\Mc$ will be a specification of an
open subset $A\subset \Xi$ on which data are given, and in this case,
the limited data operator can be written \bel{MA} \Ma f = \chia
\Mc,\ee where $\chia$ is the characteristic function of $A$ and the
product just restricts the data to the set $A$.  In the cases we
consider, the reconstruction operator is of the form
\bel{MA-reconstruction}\Mst P\Ma,\ee where $\Mst$ is an appropriate
dual or backprojection operator to $\Mc$, and this models our
reconstruction operator \eqref{eq:reconstruction operators}.

 Our next theorem tells what
multiplication by $\chia$ does to the wavefront set.  It is a special case
of Theorem 8.2.10 in \cite{Hoermander03}.

 \begin{theorem}\label{thm:WF mult} Let $u$ be a distribution and let
$A$ be a closed subset of $\,\Xi$ with nontrivial interior.  If the
\emph{non-cancellation condition} \bel{non-cancellation}\text{\rm\
$\forall\,(y,\xi)\in T^*(\Xi)$, }\ (y,\xi)\in \WF(u) \text{\rm\ iff }
(y,-\xi)\notin \WF(\chia)\ee holds, then the product $\chia u$ can be
defined as a distribution.  In this case, we have \bel{WF of a
product}\WF(\chia u) \subset \Qc(A,\WF(u)),\ee where for $A\subset \Xi$
and $W\subset T^*(\Xi)$ \bel{def:Q}\begin{aligned}\Qc(A,W) :=&
\big\{(y,\xi+\eta)\st y\in A\,, \bparen{(y,\xi)\in W\text{\rm\ or }
\xi = 0}\\&\qquad
\text{\rm\  and } \big[(y,\eta)\in \WF(\chia)\text{\rm\  or } \eta =
0\big]\big\}\,.\end{aligned}\ee
\end{theorem}

Note that the condition ``$y\in A$'' is not in \eqref{def:Q} in
H\"{o}rmander's theorem, but we  include this  because
$\chia$ is zero (and so smooth) off of $A$.  Also, note that the case
$\xi=\eta = 0$ in the definition of $\Qc$ is not allowed since the
wavefront set does not include zero vectors.

Our paradigm for proving characterizations for visible and added artifacts is given by the following procedure, cf. \cite{FrikelQuinto2014}:
\begin{enumerate}[label*=(\alph*)]

\item\label{paradigm:step1} Confirm the forward operator $\Mc$ is a
FIO and calculate its canonical relation, $C$.

\item\label{paradigm:step2} Choose the limited data set $A\subset \Xi$ and calculate
$\WF(\chia)$.

\item\label{paradigm:step3} Make sure the non-cancellation condition
\eqref{non-cancellation} holds for $\chi_A$ and $\Mc f$.  This can be
done in general by making sure it holds for $(y,\eta)\in C\circ
\paren{T^*(\Om)\smo}$. 

\item\label{paradigm:step4} Calculate $\Qc(A,C\circ \WF(f))$.

\item\label{paradigm:step5} Calculate $C^t\circ\Qc\paren{A,
C\circ\WF(f)}$ to find possible visible singularities and added
artifacts using \cite[Lemma 3.2]{FrikelQuinto2014}: \bel{final WF
containment}\WF(\Mst P \Ma f)\subset C^t\circ\Qc\paren{A,
C\circ\WF(f)}.\ee

\end{enumerate}

\section{Characterization of artifacts} \label{sec:characterization}
The first proposition provides the microlocal properties of $\Rmu$ and
$\Rstm$.

\begin{proposition}
\label{prop:canonical relations} If $\mu$ is nowhere zero, then the
generalized Radon transform $\Rmu$ is an elliptic Fourier integral
operator associated to the canonical relation \begin{multline}
\label{def:C}
		C = \{
	((\phi,s),\alpha\bparen{-x\cdot\th^\bot(\phi)\d\phi+\d s}; 
x,\alpha\th(\phi)\d x)\st \\
		 (\phi,s)\in[0,2\pi]\times \R,\alpha\neq 0, x\in
		 L(\phi,s) \}.
	\end{multline}

 The dual operator $\Rstm$ is an elliptic Fourier integral operator
associated to the canonical relation $C^t$ defined in \eqref{def:At}.

Let $\Pi_R:C\to T^*(\rtwo)$ and $\Pi_L:C\to T^*(\Xi)$ be the natural
projections. Then $\Pi_L$ is an injective immersion and $\Pi_R$ is a
two-to-one immersion.  Let $(x,\xi\dx)\in T^*(\rtwo)\smo$.  Let
$\phio=\phio(\xi)$ be the unique angle in $[0,2\pi)$ with
$\xi=\norm{\xi}\th(\phio)$ and let $\phion=\phion(\xi)$ be the unique
angle in $[0,2\pi)$ with $\xi=-\norm{\xi}\th(\phio)$. Define
\bel{def:lambdas}\begin{aligned} \lamo(x,\xi) &=
\paren{\phio(\xi),x\cdot \th(\phio(\xi)), \norm{\xi}\bparen{-x\cdot
\thperp(\phio(\xi))\d \phi + \d s}}\\
\lamon(x,\xi) &= \paren{\phion(\xi),x\cdot \th(\phion(\xi)),
-\norm{\xi}\bparen{-x\cdot \thperp(\phion(\xi))\d \phi + \d
s}}.\end{aligned}\ee
The two preimages of $(x,\xi\dx)$ under $\Pi_R$ are
$(\lamo(x,\xi);x,\xi\dx)$ and $(\lamon(x,\xi);x,\xi\dx)$. 
Therefore, \bel{lambdas}\begin{aligned}C\circ\sparen{(x,\xi\dx)}&=
\sparen{\lamo(x,\xi),\lamon(x,\xi)}\\
C^t\circ\sparen{\lamo(x,\xi\dx)}&=
C^t\circ\sparen{\lamon(x,\xi)}=\sparen{(x,\xi\dx}.\end{aligned}\ee
\end{proposition}

\begin{proof}  The calculation of $C$ is well known, see e.g., 
 \cite{FrikelQuinto2013,GS1977}.  Since $\Rstm$ is the dual of
$\Rmu$, $\Rstm$ is an FIO associated to $C^t$ by the standard calculus
of FIO, e.g., \cite[Theorem 4.2.1]{Ho1971}.  That $\Pi_L:C\to
T^*(\Xi)$ is an injective immersion (The Bolker Assumption) is a
straightforward calculation \cite{GS1977, Quinto1980}.

One uses \eqref{def:C} to find the two preimages of $(x,\xi\dx)$ under
$\Pi_R:C\to T^*(\rtwo)$ using the fact that $\xi =
\norm{\xi}\th(\phio(\xi))=-\norm{\xi}\th(\phion(\xi))$.  Statement
\eqref{lambdas} follow from the observation that, if $A\subset
T^*(\rtwo)$, then $C\circ A = \Pi_L\paren{\Pi_R\inv (A)}$ (and if
$B\subset T^*(\Xi)$, then $C^t\circ B = \Pi_R\paren{\Pi_L\inv (B)}$).
\end{proof}

The next theorem provides a generalization to $\Rmu$ and arbitrary
filter $P$ of the artifact characterization in \cite{FrikelQuinto2013,
Katsevich:1997}.

\begin{theorem}
\label{thm:characterization} Let $f\in \Ec'(\rtwo)$ and let $\mu$ be a
nowhere zero smooth $2\pi-$periodic function on $\rr\times
\rtwo$. Let $P$ be a pseudodifferential operator on $\Ec'(\Xi)$ 
\begin{equation} \label{equn:characterization}
 \WF(\Rstm P\Rmuab f)\subset
\WF_{(a,b)}(f)\cup\Ac_{\{a,b\}}(f), \end{equation} where
\begin{equation}\label{def:WFab} \WF_{(a,b)}(f)=\WF(f)\cap \Vab,
  \text{\rm and}\ \Vab = \{(x,\alpha\th(\phi)\d x):
\alpha\neq0,\phi\in(a,b)\} \end{equation} is the set of visible
singularities and \begin{multline} \Ac_{(a,b)}(f) =
\{(x+t\th^\bot(\phi),\alpha\th(\phi)\d x)\st \\ \phi\in\sparen{a,b},
\alpha,t\neq0, x\in L(\phi,s), (x,\alpha\th(\phi))\in\WF(f)\}
\label{def:A}\end{multline} is the set of possible added artifacts.

 Now, assume that $\mu$ is nowhere zero and the top order symbol of
$P$ is nowhere zero modulo lower order symbols on
$\sparen{(\phi,s,\alpha[t\d \phi + \d s])\st \phi\in (a,b), s\in \rr,
t\in \rr, \alpha \neq 0}$.  Furthermore assume $b-a<\pi$.  Then,
\bel{elliptic containment} \WF_{(a,b)}(f) \subset \WF(\Rstm P\Rmuab
f).\ee
\end{theorem}

The condition $b-a<\pi$ is reasonable in limited data problems because, if
$b-a\geq \pi$, then every line is parameterized by $L(\phi,s)$ for $\phi\in
(a,b)$.

Radon transforms detect singularities conormal to the set being
integrated over (e.g., \cite{GS1977,Palamodov,Quinto93}), and the
above theorem states this relation explicitly: only singularities
$(x,\alpha\th(\phi))\in\WF(f)$ with directions in the visible angular
range, $\phi\in(a,b)$, can be reconstructed from limited angle data.
Singularities of $f$ outside of $[a,b]$ are smoothed.  However, each
singularity of $f$ at $(x,\alpha\th(\phio))$ for $\phio=a,b$ generates
a line of artifacts through $x$ and normal to $\th(\phio)$.

\begin{proof}
We use the paradigm presented in Section \ref{sec:paradigm} to prove
\eqref{equn:characterization}.  By Proposition \ref{prop:canonical
relations}, we know that $\Rmu$ is a Fourier integral operator with
the canonical relation given in \eqref{def:C}. Thus,
the step \ref{paradigm:step1} of our paradigm is carried out.

For the step \ref{paradigm:step2}, we consider $A=(a,b)\times\R$ with
$0<a<b<\pi$ and compute
\begin{equation}
\label{eq:WF hard cutoff}
	\WF(\chi_{(a,b)\times\R}) = \sparen{((\phi,s);\beta\d\phi)\st \phi\in\sparen{a,b},\beta\neq0,s\in\R}.
\end{equation}
Since $\WF(\chi_{(a,b)\times\R})$ has no $\d s$-component and at the same time the $\d s$-component of $\WF(\Rmu f)$ is always non-zero, we see that the non-cancellation condition \eqref{non-cancellation}
holds. This is step \ref{paradigm:step3} of our paradigm.  Hence, by
Theorem \ref{thm:WF mult}, the product
\begin{equation}
	\Rmuab f = \chi_{(a,b)}\cdot\Rmu f
\end{equation}
is well-defined and   
\[\WF(\Rmuab f)\subset \Qc((a,b)\times\R,C\circ\WF(f)).\] 

In the next step (cf. \ref{paradigm:step4}), we calculate
$\Qc((a,b)\times\R,C\circ\WF(f))$ using \eqref{def:Q}.  

Since the condition [$\xi=0$ and $\eta=0$] is not allowed, the set 
$\Qc((a,b)\times\R,C\circ\WF(f))$ is a union of three sets:
\begin{multline}
\label{eq:wavefront set restricted Radon transform}
	\Qc((a,b)\times\R,C\circ\WF(f)) = \bparen{(C\circ\WF(f)) \cap
	\{((\phi,s),\eta)\in T^\ast (\Xi)\st \phi\in(a,b)\}}\\ \cup
	\WF(\chi_{(a,b)}) \cup \Wab(f),
\end{multline}
where the first set (in braces) corresponds to $\xi\neq 0$, $\eta=0$,
the second to $\xi=0$, $\eta\neq 0$ and the third, $\Wab(f)$,
corresponds to $\xi\neq 0$, $\eta\neq 0$ in the definition of $\Qc$.
To calculate $\Wab(f)$ note that covectors in $C\circ \WF(f)$ are of
the form $((\phi,s); \alpha(-\delta \d \phi + \d s))$ where there
exists an $x\in L(\phi,s)$ with $(x,\alpha\theta(\phi))\in \WF(f)$ and
where $\delta = x\cdot \thperp(\phi)$.  Also, $\eta\neq 0$
corresponds to covectors in $\WF(\chi_{(a,b)})$, which are of the form
$((\phi,s); \beta\d\phi)$ where $\phi\in \sparen{a,b}$ and $\beta\neq
0$.  Adding these vectors for the same base point, one sees that the
covector $((\phi,s); (\beta -\alpha\delta)\d\phi + \alpha\d s)$ is in
$\Wab(f)$.  Since $\beta$ is arbitrary, one can write 
\begin{multline}\label{Wab1}
  	\Wab(f) = \{ ((\phi,s); \nu\d\phi +\alpha\d s)\st \\
\nu\in \rr, \alpha\neq0, \phi\in\{a,b\}, \exists x\in L(\phi,s)\st
(x,\alpha\th(\phi))\in\WF(f)\}.
\end{multline}

To accomplish the step \ref{paradigm:step5} in our paradigm, we let $P$ be
a pseudodifferential operator.  Then, by containment \eqref{final WF
  containment},
\[\WF(\Rstm P\Rmuab f)\subset
C^t\circ\Qc((a,b)\times\R,C\circ \WF(f)).\]

We now compute $C^t\circ\Qc((a,b)\times\R,C\circ \WF(f))$. Using
\eqref{eq:wavefront set restricted Radon transform} and the
composition rules, first observe that
\begin{multline}
\label{eq:three parts of the wavefront set}
	C^t \circ \Qc((a,b)\times\R,C\circ \WF(f)) = C^t  \circ
	\big[(C\circ\WF(f)) \cap \{((\phi,s),\eta)\in T^\ast (\Xi)
\st \phi\in(a,b)\}\big]\\ 
	  \cup C^t  \circ \WF(\chi_{(a,b)})
	  \cup C^t  \circ \Wab(f) .
\end{multline}

We examine the three terms of the equation \eqref{eq:three parts of
the wavefront set} separately. First, we get
\begin{multline}
	C^t  \circ \big[(C\circ\WF(f)) \cap
	\{((\phi,s),\eta)\in T^\ast (\Xi)\st \phi\in(a,b)\}\big] \\
	=\big[(C^t  \circ C )\circ\WF(f))\big] \cap \big[C^t \circ
	\{((\phi,s),\eta)\in T^\ast (\Xi)\st \phi\in(a,b)\}\big].
\end{multline}
It is not hard to see that $C^t\circ C = \Delta :=\sparen{(x,\xi \d
x;x,\xi\d x)\st (x,\xi\d x)\in T^\ast\R^2}$ and
$\Delta\circ\WF(f)=\WF(f)$. Moreover, \[C^t \circ
\{((\phi,s),\eta)\in T^\ast (\Xi)\st \phi\in(a,b)\} = \Vc_{(a,b)}.\] Hence, the
first set in \eqref{eq:three parts of the wavefront set} is equal to
the set of visible singularities \eqref{def:WFab} \[\WF_{(a,b)}
(f)=\WF(f)\cap \Vc_{(a,b)}.\]

For the second set in \eqref{eq:three parts of the wavefront set}
observe that $C^t\circ \WF(\chi_{(a,b)})=\emptyset$ since the $\d
s$-components of covectors in $\WF(\chi_{(a,b)})$ is zero and the $\d
s$-components of covectors in $C^t$ is always non-zero. 

Finally, we consider the set $C^t\circ \Wab(f) $. Let \[\gamma = ((\phi,s);
\nu\d\phi+\alpha\d s)\in W_{\{a,b\}}(f),\] then $\nu \in \rr,\
\alpha\neq0$, $\phi\in\{a,b\}$, $s\in \rr$, and there is a $x\in
L(\phi,s)$ such that $(x,\alpha\th(\phi))\in\WF(f)$. 
By the definition of composition, \eqref{def:circ},
 \[C^t\circ \sparen{\gamma}=\sparen{(\tx,\alpha \th(\phi)\d x)\st
 (\tx,\alpha \th(\phi)\d x;\gamma)\in C^t}\] where, by the definition
of $C^t$, (\eqref{def:C} with the coordinates switched), $\tx\in
L(\phi,s)$ so $s=\tx\cdot \th(\phi)$.  Let $t=-\nu/\alpha$.  Since
$\nu$ is arbitrary, $t$ is arbitrary.  Again by the definition of
$C^t$, $t=-\nu/\alpha = \tx\cdot \thperp(\phi)$, so the point $\tx =
s\th(\phi) + (-\nu/\alpha)\thperp(\phi)$ is an arbitrary point in
$L(\phi,s)$.  Therefore, for any $\tx\in L(\phi,s)$, the covector
$(\tx,\alpha \th(\phi)\d x)\in C^t\circ \Wab(f)$. Thus, the third set
in \eqref{eq:three parts of the wavefront set} is the set of possible
added singularities given by \eqref{def:A}.

\medskip

Containment \eqref{elliptic containment} is proven using Corollary
\ref{cor:reduction2} from the next section.  Let $(x,\xi\dx) \in \WF(f)\cap
\Vc_{(a,b)}$.  Then, one of the angles $\phio(\xi)$ or $\phion(\xi)$
(defined in Proposition \ref{prop:canonical relations}) is in $(a,b)$ and
the other one is not since $b-a<\pi$.  Without loss of generality, assume
$\phio(\xi)\in (a,b)$.

Let $\vp$ be a cutoff function in $\phi$ that is supported in $(a,b)$
and equal to one in a smaller neighborhood $(a',b')$ of $\phi'$.  We
will define $\Kvp$ as the multiplication operator $\Kvp g(\phi,s) =
\vp(\phi)g(\phi,s)$.

Let $g_1=P \Kvp \Rmu(f)$ and $g_2 = P\bparen{\chi_{(a,b)}-\vp}
\Rmu(f)$.
By Corollary \ref{cor:symbol elliptic} part \ref{b-a}, the symbol of
$\Rstm P \Kvp \Rmu$ is elliptic on $\Vc_{(a',b')}$ and so at
$(x,\xi\dx)$.  Therefore, $(x,\xi\dx)\in \WF(\Rstm g_1 )$.  
We now show $(x,\xi\dx)\notin \WF\paren{\Rstm g_2}$.  Because
$\chi_{(a,b)}- \vp$ is zero on $(a',b')$, $\bparen{\chi_{(a,b)}-
\vp}\Rmu f$ is zero on $(a',b')\times \rr$.  Therefore, $g_2=
P\bparen{\chi_{(a,b)}- \vp} \Rmu(f))$ is smooth on $(a',b')\times
\rr$, and since $\phio(\xi)\in (a',b')$, $\lamo(x,\xi)\notin
\WF(g_2)$.  Since $b-a<\pi$, $\phion(\xi)\notin(a,b)$, so $g_2$ is
smooth near $\phion(\xi)$. This implies that $\lamon(x,\xi)\notin
\WF(g_2)$.  Using the H\"ormander-Sato Lemma \ref{thm:HS}, $\WF(\Rstm
g_2)\subset C^t\circ \WF (g_2)$, so, by \eqref{lambdas} the only two
covectors, $\lamo(x,\xi)$ and $\lamon(x,\xi)$, that can contribute to
wavefront of $\Rstm g_2$ at $(x,\xi\dx)$ are not in $\WF(g_2)$ so
$(x,\xi\dx)\notin \WF(\Rstm g_2)$.  

Therefore, $(x,\xi\dx)\in \WF(\Rstm g_1+\Rstm g_2)=WF(\Lvp f)$, and
this proves the final part of the theorem.
\end{proof}


\section{Reduction of artifacts}
\label{sec:reduction}

The singularity reduction method replaces the sharp cutoff $\chiab$ by
a smooth cutoff.  Let $\vp$ be a smooth cutoff function supported in
$(a,b)$ and equal to one on a proper subinterval $(a',b')$, and
replace $\chiab$ by $\vp$ in the reconstruction operator.  Then the
artifact-reduced reconstruction operator is \bel{def:L} \Lvp f = \Rstm
P\Kvp \Rmu f \ \text{ where } \ \Kvp g = \vp g.\ee This method was
analyzed for the lambda filter $P=-d^2/ds^2$ and the FBP filter
$P=\sqrt{-d^2/ds^2}$ and with $R_1$ in \cite{FrikelQuinto2013} and
with $\Rmu$ in \cite{Katsevich:1997,KLM}. Our theorems provide
generalization to arbitrary filters $P$, and they provide the symbol
of $\Lvp$ in general with proof.

\begin{theorem}\label{thm:reduction1}  
  Let $\mu$ be a smooth measure and let $\vp$ be a smooth function
  supported in $(a,b)$ and equal to $1$ on the proper subinterval
  $(a',b')$.  Then \bel{right reduction containment} \WF(\Lvp (f))\subset
  \WF_{(a,b)}(f).\ee The top order symbol of $\Lvp$ is
  \bel{symbol}\begin{aligned} \sigma(\Lvp)(x,\xi\dx) =&
    \frac{2\pi}{\norm{\xi}} \Big[
    \vp(\phio(\xi))p(\lamo(x,\xi))\mu^2(\phio(\xi),x)\\
    &\qquad+\vp(\phion(\xi))p(\lamon(x,\xi))\mu^2(\phion(\xi),x)
    \Big] \end{aligned}\ee where $P$ is a pseudodifferential operator on
  $\Ec'(\Xi)$ and the notation is given in \eqref{def:lambdas}.
\end{theorem}

If $\nu$ is a smooth weight and $\Rstm$ is replaced by $R_{\nu}^*$,
then the $\mu$ factor in \eqref{symbol} is replaced by $\nu \mu$.

\begin{corollary}\label{cor:reduction2}   Let $\vp$
be a nonnegative smooth function supported on $(a,b)$ and equal to $1$
on a subinterval $(a',b')$.  Assume the symbol $\sigma(\Lvp)$ in
\eqref{symbol} is nowhere zero modulo lower order symbols.  Then,
\bel{left reduction containment}\WF_{(a',b')}(f) \subset \WF(\Lvp
(f)).\ee 
\end{corollary}

This theorem shows that as long as $P$ is well-chosen, most visible
wavefront directions (those in $\WF_{(a',b')}(f)$) are visible using the
artifact reduced operator $\Lvp$ and artifacts are not added since $\WF(
\Lvp(f))$ is contained in $\WF_{(a,b)}(f)$.  The proof follows from the
ellipticity assumption in the corollary using, e.g., \cite[Prop.\
6.9]{Treves:1980vf}.

Our next corollary provides specific cases in which the theorem can be
applied.

\begin{corollary}\label{cor:symbol elliptic}
Let $\vp$ be a nonnegative function supported in $(a,b)$ and equal to
$1$ on the subinterval $(a',b')$.  Let \[\Ac=\sparen{(\phi,s,\alpha[t\d
\phi + \d s])\st \phi\in (a',b'), s\in \rr, t\in \rr, \alpha \neq 0}.\]

Then $\Lvp =\Rstm \Kvp P \Rmu$ is elliptic on $\Vc_{(a',b')}$ (therefore
\eqref{left reduction containment} holds) when either of the following
conditions hold for $\mu$ and $P$:

\begin{enumerate} 

\item\label{same sign}  $\mu$ is real and nowhere zero and the
top order symbol $\sigma(P)=p$ is real and nonzero on $\Ac$, or

\item\label{b-a}  $b-a<\pi$ and $\mu$ is nowhere zero and 
$p$ is elliptic on $\Ac$.

\end{enumerate}
\end{corollary}

Condition \ref{same sign} holds, for example, if $P=-d^2/ds^2$, the
filter in Lambda tomography, or $P=\sqrt{-d^2/ds^2}$, the filter in
FBP because, in both cases, the symbol is positive on $\Ac$ (e.g.,
$\sigma(\sqrt{-d^2/ds^2})(\phi,s,\beta\d\phi + \alpha\d
s)=\abs{\alpha}$), and our theorem can be applied to these operators.

If $b-a<\pi$ and $P=d/ds$, then condition \ref{b-a} holds since the
symbol of $d/ds$ is nowhere zero on $\Ac$.  Thus, $\Lvp$ is elliptic
on $\Vc_{(a',b')}$.  However, if $b-a>\pi$, ellipticity of $P$ is not
sufficient for ellipticity of $\Lvp$.  For example, consider the full
data problem for the classical transform $R_1$, then
$\sigma(P)(\phi,s,\beta\d\phi+\alpha\d s) = \alpha $ changes sign on
$\Ac$ and the operator $R_1^*(d/ds R_1)=0$ by symmetry.

\begin{proof}[Proof of Theorem \ref{thm:reduction1}]
  We use the notation, conventions, and symbol calculation in \cite[Theorem
  3.1]{Quinto1980}.  Recall that $\Pi_R:C\to T^*(\rtwo)$ and $\Pi_L:C\to
  T^*(\Xi)$ are the natural projections.  Equation (14) in
  \cite{Quinto1980} and the discussion below it give the symbol of $\Rmu$
  as the half density \bel{symbol
    Rmu}\sigma(\Rmu)=\frac{(2\pi)^{1/2}\mu(\phi,x)d \phi\, d x\,
    \sqrt{dw\,d\eta}}{\sqrt{d\phi\, d s \, d x
    }\,\Pi_R^*(\abs{\sigma_\rtwo})}\ee where $\abs{\sigma_\rtwo}$ is the
  density from the canonical symplectic form on $T^*(\rtwo)$ and
  $\Pi_R^*(\abs{\sigma_\rtwo})$ is its pull back to $C$.  Also,
  $Z=\sparen{(\phi,x\cdot \th(\phi),x)\st \phi\in [0,2\pi), x\in \rtwo}$ is
  the set in $\Xi\times \rtwo$ over which the Schwartz kernel of $\Rmu$
  integrates, and $z=(\phi,x\cdot\th(\phi),x)$ and $w=x\cdot\th(\phi)-s$
  give coordinates on $\Xi\times \rtwo$. Then, the \emph{measure} on $Z$
  associated to $\Rmu$ is $\mu(\phi,x)d\phi\,d x $ (see equation (16) in
  \cite{Quinto1980}).  Finally $\eta$ is the fiber coordinate in the
  conormal bundle of $Z$.  An analogous argument shows that the symbol of
  $\Rstm$ is given by \bel{symbol
    Rstm}\sigma(\Rstm)=\frac{(2\pi)^{1/2}\mu(\phi,x)d \phi\, d x
    \sqrt{dw\,d\eta}}{\sqrt{ d\phi\, d s\, d
      x}\,\Pi_L^*(\abs{\sigma_\Xi})}.\ee
The pseudodifferential operator $P\Kvp$ has symbol
$\vp(\phi)p(\phi,s,\gamma)$ (where $\gamma\in T^*_{(\phi,s)}(\Xi)$) so
$P\Kvp \Rmu$ is a standard smooth FIO and its top order symbol is
\[\sigma(P \Kvp
\Rmu)=\frac{(2\pi)^{1/2}p(\phi,s,\gamma)\vp(\phi)\mu(\phi,x)d \phi\, d x
\sqrt{dw\,d\eta}}{\sqrt{d\phi \,d s \,d x }\,\Pi_R^*(\abs{\sigma_\rtwo})}\] when evaluated at
covectors on $C$.

 Let $(x,\xi\dx)\in T^*(\rtwo)\smo$.  To calculate the symbol of the
composition of $\Rstm$ with $P\Kvp \Rmu$ one uses the note at the top
of p.\ 338 of \cite{Quinto1980}: since the projection $\Pi_R:C\to
T^*(\rtwo)\smo$ is two-to-one, the symbol of $\Rstm P \Kvp \Rmu$ at
$(x,\xi\dx)\in T^*(\rtwo)$ is the sum of the product
$\sigma(\Rstm)\cdot\sigma(P\Kvp \Rmu)$ at the two preimages.  By
Proposition \ref{prop:canonical relations}, those preimages, given by
$\Pi_R\inv(x,\xi\dx)$, are the two covectors
\[(\lamo(x,\xi);x,\xi\dx)\\ \text{ and }\\
(\lamon(x,\xi);x,\xi\dx).\]  

Under the conventions of \cite{Quinto1980}, the symbol of $\Rstm P\Kvp
\Rmu$ at $(x,\xi\dx)$ is the sum
\bel{symbol-start}\begin{aligned}\sigma(\Rstm P \Kvp \Rmu)&(x,\xi\dx)
= \sparen{\frac{2\pi (d \phi\, \d x)^2 d w\,
d\eta}{d \phi\, d s\,d x\,\Pi_R^*(\abs{\sigma_\rtwo})\,\Pi^*_L(\abs{\sigma_\Xi})}}\\
&\quad\times 
\big[\vp(\phio(\xi))\mu^2(\phio(\xi),x)p(\lamo(x,\xi))\\
&\quad\quad+\vp(\phion(\xi))\mu^2(\phion(\xi),x)p(\lamon(x,\xi))\big]\end{aligned}\ee
Now, \cite[Lemma 3.2]{Quinto1980} shows, for the Radon line transform,
that the term on the top right in braces in \eqref{symbol-start} can
be  simplified to equal to $2\pi/\norm{\xi}$.  Putting this
into \eqref{symbol-start} proves the symbol calculation
\eqref{symbol}.
\end{proof}

\begin{proof}[Proof of Corollary \ref{cor:symbol elliptic}]
In each case, we will show that $\sigma(\Lvp)$ is elliptic on
$\Vc_{(a',b')}$.  Let $(x,\xi\dx)\in \Vc_{(a',b')}$, then either
$\phio(\xi)$ or $\phion(\xi)$ or both are in $(a',b')$. Without loss
of generality, we assume $\phio(\xi)\in (a',b')$.  Therefore,
$\vp(\phio(\xi))=1$.  

In case \ref{same sign} we assume $\mu$ is real and nowhere zero and
the top order symbol of $P$, $\sigma(P)=p$, is real and nowhere zero
on $\Ac$.  Therefore $p$ is either always positive or always negative
on $\Ac$.  Since $\vp=1$ on $(a',b')$ and $\mu^2>0$, at least the
first term in brackets in \eqref{symbol} (the one containing $\phio$)
is nonzero.  The second term (containing $\phion(\xi)$) either has the
same sign as this term (since the sign of $p$ does not change) or is
zero (if $\phion(\xi)\notin \supp(\vp))$.  Therefore the sum is
nonzero and so the symbol of $\Lvp$ is elliptic on $\Vc_{(a',b')}$.  

In case \eqref{b-a}, since $b-a<\pi$ and $\phio(\xi)\in (a',b')$,
$\phion(\xi)\notin (a',b')$.  Therefore, only one term in brackets in
\eqref{symbol} is nonzero.  Therefore, the symbol is elliptic on
$\Vc{(a',b')}$.
\end{proof}

\section*{Acknowledgments} The authors thank Frank Filbir for
encouraging this collaboration.  They thank Adel Faridani and
Alexander Katsevich for pointing out important references.  The second author is supported by NSF grant DMS 1311558

\bibliographystyle{AIMS}
\bibliography{references}
\end{document}